\documentclass[a4paper]{amsart}

\title{Sum-free sets which are closed under multiplicative inverses}
\author{Katherine Benjamin}
\address{Mathematical Institute, University of Oxford, Woodstock Road,
Oxford OX2 6GG, United Kingdom}
\email{katherine.benjamin@stcatz.ox.ac.uk}

\newif\ifdetails

\usepackage{amssymb, mathtools}

\usepackage{pgfplots}
\pgfplotsset{compat=1.12}

\usepackage{caption}
\def\F{\mathbb{F}}
\def\R{\mathbb{R}}
\def\C{\mathbb{C}}
\def\E{\mathbb{E}}
\def\eps{\varepsilon}
\def\from{\colon}

\newcommand{\wh}{\widehat}
\DeclareMathOperator{\Tr}{Tr}
\renewcommand{\Re}{\operatorname{Re}}

\DeclarePairedDelimiter\abs{\lvert}{\rvert}

\newcommand{\case}[1]{\medskip \paragraph{\textbf{#1}}}

\newtheorem{theorem}{Theorem}[section]
\newtheorem{lemma}[theorem]{Lemma}
\newtheorem{corollary}[theorem]{Corollary}
\newtheorem{proposition}[theorem]{Proposition}
\newtheorem{conjecture}[theorem]{Conjecture}

\theoremstyle{remark}
\newtheorem*{remarks}{Remarks}
\begin{document}

\begin{abstract}
    Let $A$ be a subset of a finite field $\F$. When $\F$ has prime order, we
    show that there is an absolute constant $c > 0$ such that, if $A$ is both
    sum-free and equal to the set of its multiplicative inverses, then
    $\abs{A} <  (0.25 - c)\abs{\F} + o(\abs{\F})$ as $\abs{\F} \rightarrow
    \infty$. We contrast this with the result that such sets exist with size at
    least $0.25\abs{\F} - o(\abs{\F})$ when $\F$ has characteristic $2$.
\end{abstract}

\maketitle

\section{Introduction}\label{sec:introduction}

Let $A$ be a subset of a finite field $\F$. We say $A$ is \emph{sum-free} if
$A \cap (A + A) = \varnothing$, where \[
    A + A \coloneqq \{a + b : a, b \in A\}
.\] 
We say $A$ is \emph{closed under (multiplicative) inverses} if $0 \not\in A$ and
$A = A^{-1}$, where \[
    A^{-1} \coloneqq \{ a^{-1} : a \in A \}
.\] 
In this paper, we study sets which are both sum-free and closed under inverses.

When $\F$ has prime order, a simple application of the Cauchy-Davenport 
inequality (see e.g.\ \cite[Theorem~5.4]{tv2006}) shows that 
$\abs{A} \leq  (\abs{\F}+1) / 3$ when $A$ is sum-free. Lev showed in
\cite{lev2006} that when $\abs{A}$ is close to $\abs{\F} / 3$, $A$ is similar
in structure to an arithmetic progression, and therefore unlikely to be closed
under inverses. So, we might expect $\abs{A}$ to be smaller than 
$\abs{\F} / 3$ if $A$ is also closed under inverses. 

In this direction, Bienvenu et al.\ showed in 
\cite[Corollary~5.1]{bhs2019} that $\abs{A} < 0.3051\abs{\F} + o(\abs{\F})$ as 
$\abs{\F} \rightarrow \infty$. We offer the following improvement on this:

\begin{theorem}\label{thm:bound}
    There is an absolute constant $c > 0$ so that if $\F$ is a field of prime
    order and $A \subseteq \F^{*}$ is sum-free and closed under inverses then
    $\abs{A} < (0.25-c)\abs{\F} + o(\abs{\F})$ as $\abs{\F} \rightarrow \infty$.
\end{theorem}

This is in contrast to fields of characteristic $2$, where we show:

\begin{proposition}\label{prop:char2}
    If $\F$ is a field of characteristic $2$ then there exists
    $A \subseteq \F^*$ which is both sum-free and closed under inverses, such
    that $\abs{A} = 0.25\abs{\F} + o(\abs{\F})$ as $\abs{\F} \rightarrow \infty$.
\end{proposition}

Write $\mu(\F)$ for the density $\abs{A}/\abs{\F}$ of the largest
$A \subseteq \F$ which is both sum-free and closed under inverses.
Theorem \ref{thm:bound} says that $\mu(\F_p) \leq 0.25 - c + o(1)$,
whereas Proposition \ref{prop:char2} says that
$\mu(\F_{2^{n}}) \geq 0.25 - o(1)$. So we can deduce that:

\begin{corollary}
    The limit $\lim_{\abs{\F} \rightarrow \infty}{ \mu (\F)}$ does not exist.
\end{corollary}

The rest of the paper is structured as follows. In Section \ref{sec:definitions}
we recall some basic definitions of Fourier analysis, and establish some
notation. In Section \ref{sec:prime} we consider fields of prime order. We
establish some Fourier analytic results and use them to prove Theorem
\ref{thm:bound}. Then, in Section \ref{sec:char2} we consider fields of even
characteristic, and prove Proposition \ref{prop:char2}. In Section
\ref{sec:remarks} we make some final remarks.

\section{Notation and definitions from Fourier analysis} \label{sec:definitions}

Let $\F$ be a finite field. We recall some basic definitions from Fourier
analysis (see e.g.\ \cite[Section~4]{tv2006} or \cite[Section 1.1]{wolf2015}). 

If $X \subseteq \F$ is non-empty and $f \from X \to \C$ is any function, we
define the \emph{mean} \[
    \underset{x \in X}{\E}[f(x)] \coloneqq \frac{1}{\abs{X}}\sum_{x \in X}f(x)
.\] 
We will also write \[
\E[f] = \underset{x}{\E}[f(x)] = \underset{x \in \F}{\E}[f(x)]
\]
when it is unambiguous to do so. We denote by $1_X$ the indicator function
\[
    1_X(x) \coloneqq 
        \begin{cases}
            1 & \text{if } x \in X, \\
            0 & \text{otherwise.}
        \end{cases}
\]

When $\F$ has prime order $p$ we can view the set of functions $\F \to \C$ as a
Hilbert space by equipping it with the inner product \[
    \langle f, g \rangle \coloneqq \E[f \overline{g}]
.\] 
Write $e(\theta) = \exp(i \theta)$ for the exponential map $\R \to \C$. For each
$r \in \F$, define the \emph{character}\footnote{
    We follow the notation of \cite{tv2006}. It is also common to write
    $e_p(x) = e(2 \pi x / p)$.
} $e_r \from \F \to \C$ by $e_r(x) \coloneqq e(2 \pi rx / p)$.
The characters enjoy the following orthogonality property: \[
    \langle e_r, e_s \rangle =
        \begin{cases}
            1 & \text{if $r = s$} \\
            0 & \text{otherwise.}
        \end{cases}
\]
This motivates the definition of the \emph{Fourier coefficient of $f$ at $r$} as
\[
    \wh{f}(r) \coloneqq \langle f, e_r \rangle 
.\] 
\emph{Parseval's identity} is then
\begin{equation*}\label{eq:parseval}
    \E[\abs{f}^2] = \sum_{r \in \F}\bigl|\wh{f}(r)\bigr|^{2}.
\end{equation*}

\section{Fields of prime order} \label{sec:prime}

The goal of this section is the prove Theorem \ref{thm:bound}. Let $\F = \F_p$
be a field of prime order $p>2$. Let $A$ be a subset of $\F^*$, not necessarily
sum-free or closed under inverses, with density $\alpha = \abs{A} / p$. We fix
some $0 < \alpha_0 < 0.25$ and assume $\alpha \geq \alpha_0$, since otherwise
Theorem \ref{thm:bound} is immediate.

Order the elements $r_1, \dots, r_{(p-1) / 2}$ of the interval
$\{1, \dots, (p-1) / 2\} \subseteq \F$ so that
$\delta_1 \geq  \dots \geq \delta_{(p-1) / 2}$, where
$\bigl|\wh{1_A}(r_i)\bigr| = \delta_i\alpha$. Note that \[ 
    \F^{*} = \{r_1, \dots, r_{(p-1) / 2} \} \cup \{-r_1, \dots, -r_{(p-1) / 2}\}
\] and that $\wh{1_A}(-r_i) = \overline{\wh{1_A}(r_i)}$ for each $i$. We will
also write $\theta_1 \in [0, 2\pi)$ for the argument of $\wh{1_A}(r_1)$, so that
$\wh{1_A}(r_1) = (\delta_1\alpha) e(\theta_1)$ and
$\wh{1_A}(r_1) + \wh{1_A}(-r_1) = 2\delta_1\alpha\cos\theta_1$.

\subsection{Properties of sum-free sets}

We begin by recalling a standard identity, which can be derived by considering
the convolution $1_A * 1_A$ (see e.g.\ \cite[p.~153]{tv2006}).

\begin{proposition} \label{prop:sumofcubes}
    If $A$ is sum-free then
    \begin{equation*}
        \alpha^{3} + \sum_{r\neq 0} \bigl|\wh{1_A}(r)\bigr|^{2} \wh{1_A}(r) = 0.
    \end{equation*}
\end{proposition}

In fact, this sum is dominated by its largest terms.

\begin{lemma} \label{lemma:taillemma}
    Let $k$ be a positive integer. For any $p$ such that $k < (p-1) / 2$, if
    $A \subseteq \F_p$ then \[
        \sum_{i> k }\delta_i^3 \rightarrow 0
    \] 
    as $k \rightarrow \infty$, uniformly in $A$ provided $\alpha \geq \alpha_0$.
\end{lemma}
\begin{proof}
    From Parseval's identity we know
    \[
        \alpha^{2} + 2\alpha^{2}\sum_{i\geq 1}\delta_i^2 \leq \alpha
    ,\] whence, looking at the first $k$ terms of the sum,
    \[
        \delta_k^2 \leq \frac{1-\alpha}{2k\alpha}.
    \]
    So
    \[
        \sum_{i>k}\delta_i^3 
        \leq \delta_k\sum_{i>k}\delta_i^2
        \leq k^{- 1 / 2}\left( \frac{1-\alpha}{2\alpha} \right)^{{3}/{2}}
        \leq k^{- 1 / 2}\left( \frac{1 - \alpha_0}{2\alpha_0} \right)^{{3} /{2}}
        \rightarrow 0
    .\] 
\end{proof}

\begin{corollary} \label{cor:sumdeltacubes}
    If $A$ is sum-free then
    \[
        \sum_{i=1}^{k} \delta_i^{3}
        \geq \delta_1^{3}\abs{\cos\theta_1} + \sum_{i=2}^{k}\delta_i^{3}
        \geq \frac{1}{2} - o_{k \rightarrow \infty}(1)
    ,\] where the error is uniform in $A$ provided $\alpha \geq \alpha_0$.
\end{corollary}

\begin{proof}
    The first inequality is immediate. For the second, we begin with Proposition
    \ref{prop:sumofcubes} and make two
    applications of the triangle inequality.
    \begin{align*}
        \alpha^3 
        &= \Bigl|\sum_{r \neq 0} \bigr|\wh{1_A}(r)\bigl|^2 \wh{1_A}(r)\Bigr| \\
        &= \Bigl| \sum_{i=1}^{(p-1)/2} \delta_i^2\alpha^2
            \left(\wh{1_A}(r_i) + \wh{1_A}(-r_i)\right) \Bigr| \\
        &\leq \sum_{i=1}^{ (p-1) / 2} \delta_i^2\alpha^2
            \left| \wh{1_A}(r_i) + \wh{1_A}(-r_i) \right| \\
        &\leq \delta_1^2\alpha^2\abs{2\delta_1\alpha\cos\theta_1}
            + \sum_{i=2}^{ (p-1) / 2}
                \delta_i^2\alpha^2\left( \bigl|\wh{1_A}(r_i)\bigr|
            + \bigl|\wh{1_A}(-r_i)\bigr| \right) \\
        &= 2\delta_1^3\alpha^3\abs{\cos\theta_1}
            + \sum_{i=2}^{ (p-1) / 2} 2\delta_i^3\alpha^3    
    \end{align*}
    Now divide through by $2\alpha^3$ and apply Lemma \ref{lemma:taillemma}.
\end{proof}

Another corollary of Proposition \ref{prop:sumofcubes} gives bounds on
$\alpha$ in terms of the sizes of the largest two Fourier coefficients.
The first, which considers only $\delta_1$, is standard
(\emph{c.f.~}\cite[p.~226]{lev2006}). The second is stronger when
$\delta_2$ is small compared to $\delta_1$.

\begin{corollary} \label{cor:sumfreebound}
    If $A$ is sum-free then \[
        \alpha \leq \frac{\delta_1}{1 + \delta_1}
    .\] Moreover, if $1 + \delta_2 + 2\delta_1^{2}\delta_2 - 2\delta_1^{3} > 0$
    then \[
        \alpha \leq  \frac{\delta_2}
                    {1 + \delta_2 + 2 \delta_1^{2}\delta_2 - 2\delta_1^{3}}
    .\]
\end{corollary}

\begin{proof} 
    We prove the second bound. The first is proved similarly. We begin with
    Proposition \ref{prop:sumofcubes}: 
        \begin{align*}
            \alpha^3 
            &= \Bigl|\sum_{r\neq0}\bigl|\wh{1_A}(r)\bigr|^2\wh{1_A}(r)\Bigr| \\
            &\leq 2\delta_1^3\alpha^3
                + \Bigl|\sum_{r \neq 0, \pm r_1}
                \bigl|\wh{1_A}(r)\bigr|^2\wh{1_A}(r) \Bigr| \\
            &\leq 2\delta_1^3\alpha^3 + \delta_2\alpha \sum_{r \neq0, \pm r_1}
                \bigl|\wh{1_A}(r)\bigr|^2 \\
            &= 2\delta_1^3\alpha^3
                + \delta_2\alpha
                    \left(\alpha - \alpha^2 - 2\delta_1^2\alpha^2\right).
        \end{align*}
        To get the final step here we use Parseval's identity.
        Now rearrange to find \[
            \alpha
            \left(1 + \delta_2 + 2\delta_1^{2}\delta_2 - 2\delta_1^{3}\right)
            \leq  \delta_2
        \] 
        and apply the hypothesis.
\end{proof}

\subsection{Properties of sets which are closed under inverses}

To exploit the fact that $A = A^{-1}$ we will make use of the following result
from \cite[Proposition~1]{bombieri1971}, which can be thought of as a version
of Bessel's inequality for vectors which are `almost orthogonal'.

\begin{lemma} \label{lemma:almostorthogonal}
    Let $H$ be a Hilbert space with inner product $\langle \, , \rangle$. Then
    for any $f, \varphi_1, \dots, \varphi_M \in H$ we have the inequality \[
        \lVert f \rVert^2
        \geq \sum_{i=1}^{M}
            \frac{\abs{\langle f, \varphi_i\rangle}^2}
                 {\sum_{j=1}^M\abs{\langle \varphi_i, \varphi_j\rangle}}
    .\] 
\end{lemma}

We also recall Weil's estimate for Kloosterman sums \cite[p.~207]{weil1948}.

\begin{lemma}[Weil's estimate] \label{lemma:kloosterman}
    If $p$ is prime and $a,b$ are integers with $ab \neq 0$ then \[
        \Bigl|\sum_{x \in \F_p^*} e_a(x)e_b(x^{-1})\Bigr| \leq 2\sqrt{p}
    .\] 
\end{lemma}

We arrive at a useful bound on the size of a set which is closed under inverses.

\begin{proposition} \label{prop:selfinversebound}
    Suppose $A = A^{-1}$ and let $m\geq 0$.
    Suppose $s_1, \dots, s_m$ are distinct elements of $\F_p^*$ with
    $\bigl|\wh{1_A}(s_i)\bigr| = \lambda_i\alpha$. Then \[
        \alpha \leq \frac{1}{1 + 2\sum_{i=1}^{m}\lambda_i^{2}}
           + O\left(m / \sqrt{p}\right)
    .\]

    Moreover, if $k\geq 0$ then we have the bound \[
        \alpha \leq \frac{1}{1 + 4\sum_{i=1}^{k}\delta_i^{2}}
            + O\left(k / \sqrt{p}\right)
    .\]
\end{proposition}

\begin{proof}

Define $s_0 \coloneqq 0$, and so $\lambda_0 = 1$. For each $i$ define
$\varphi_i \coloneqq e_{s_i}$ and, if $i>0$,
$\psi_i(x) \coloneqq \varphi_i(x^{-1})$, with the convention that $0^{-1} = 0$.
We aim to apply Lemma \ref{lemma:almostorthogonal} to $1_A$ and these
`almost orthogonal' functions. For $i\geq 0$ and $j > 0$ we have 
\[
    \abs{\langle \varphi_i, \psi_j \rangle}
    = \frac{1}{p}\Bigl|
        \sum_{x \in \F_p} e_{s_i}(x)\overline{e_{s_j}(x^{-1})} \Bigr|
    = \frac{1}{p}\Bigl|\sum_{x \in \F_p} e_{s_i}(x){e_{-s_j}(x^{-1})}\Bigr|
    \leq \frac{1 + 2\sqrt{p}}{p}
\] by Weil's bound. Also, using the fact that the characters are orthonormal, we have \[
    \langle \psi_i, \psi_j \rangle
    = \underset{x}{\E}\left[\varphi_i(x^{-1})\overline{ \varphi_j(x^{-1})}\right]
    = \underset{x}{\E}\left[ \varphi_i(x) \overline{ \varphi_j(x)} \right]
    = \langle \varphi_i, \varphi _j \rangle
    = \begin{cases}
        1 & \text{if $i = j$,} \\
        0 &\text{otherwise.}
      \end{cases}
\] Finally, \[
    \abs{\langle 1_A, \psi_i \rangle}
    = \frac{1}{p}\Bigl| \sum_{a \in A} \overline{\varphi_i(a^{-1})}\Bigr|
    = \frac{1}{p} \Bigl| \sum_{a \in A} \overline{\varphi_i(a)}\Bigr|
    = \left|\langle 1_A, \varphi_i \rangle\right|
    = \bigl|\wh{1_A}(s_i)\bigr| = \lambda_i \alpha
.\] So, applying Lemma \ref{lemma:almostorthogonal}, we find
\begin{align*}
    \alpha 
    &\geq \sum_{i=0}^m \frac{\lambda_i^2\alpha^2}{1 + m\left(1 + 2\sqrt{p}\right) / p }
        + \sum_{i=1}^m
            \frac{\lambda_i^2\alpha^2}{1 + (m+1)\left(1 + 2\sqrt{p}\right) / p} \\
    &\geq \alpha^2\frac{1 + 2\sum_{i=1}^m \lambda_i^2}
        {1 + (m+1)\left(1 + 2\sqrt{p}\right) / p}
,\end{align*}
from which the result follows.

For the moreover part, take $m = 2k$ and $s_i = r_i = -s_{m-i}$ for each
$i\leq k$.
\end{proof}

\subsection{Constructing large coefficients}

If $\bigl|\wh{1_A}(r)\bigr| = \delta\alpha$ then an observation of Yudin
recorded in \cite[p.~258]{lev2001} yields the following bound on
$\bigl|\wh{1_A}(2r)\bigr|$:
\begin{equation} \label{eq:weakyudin}
    \bigl|\wh{1_A}(2r)\bigr| \geq \left(2\delta^2-1\right)\alpha
.\end{equation}
We strengthen this in two ways. First we show that, given conditions on $\delta$
and the argument $\theta$ of $\wh{1_A}(r)$, the coefficient $\wh{1_A}(2r)$ lies
in the right-half plane of $\C$. Second, we show that given some lower bound on
$\alpha$, we can obtain a slightly stronger lower bound on
$\bigl|\wh{1_A}(2r)\bigr|$. We shall prove \eqref{eq:weakyudin} along the way.
\begin{lemma} \label{lemma:positivecoeff}
    Suppose $r \neq 0$ and $\wh{1_A}(r) = (\delta \alpha)e(\theta)$. Then \[
        2\Re \wh{1_A}(2r)
        = \wh{1_A}(2r) + \wh{1_A}(-2r)
        \geq 2\alpha\left(2\delta^{2}\cos^2\theta -1\right)
    .\] Moreover, if $\alpha \geq \alpha_0 > 0$ then \[
        \bigl|\wh{1_A}(2r)\bigr|
        \geq \left(2\delta^{2} - 1 + \eps -o(1)\right)\alpha
    \] as $p \rightarrow \infty$, where the error is uniform in $A$ and
    $\eps >0$, which depends only on
    $\alpha_0$, is given by \[
        \eps = \frac{2^{9}}{3^{4} \times 5^{5}} {\alpha_0}^{4}
    .\] 
\end{lemma}

\begin{proof}
    For any $\omega \in S^{1}$, it can be seen that
    \begin{equation}\label{eq:yudin}
        \underset{x}{\E}
            \left[ 1_A(x)
                \left(\overline{\omega}e_r(x) + \omega e_{-r}(x)\right)^2
            \right]
        = 2\alpha + \omega^{2}\wh{1_A}(2r) + \overline{\omega}^{2}\wh{1_A}(-2r)
.\end{equation}
By applying Cauchy-Schwarz we can compute
\begin{align*}
    \underset{x}{\E}\left[ 1_A(x) \right] 
        \underset{x}{\E}\left[ 1_A(x)
            \left(\overline{\omega}e_r(x) + \omega e_{-r}(x)\right)^2
        \right]
    &\geq \underset{x}{\E}\left[ 1_A(x)
        \left(\overline{\omega} e_r(x) + \omega e_{-r}(x) \right)
    \right]^2 \\
    &= \left( \omega\wh{1_A}(r) + \overline{\omega}\wh{1_A}(-r) \right)^2
.\end{align*}
Setting $\omega = 1$ and substituting in \eqref{eq:yudin} then gives
\[
    \alpha\left(2\alpha + \wh{1_A}(2r) + \wh{1_A}(-2r)\right)
    \geq \left(\wh{1_A}(r) + \wh{1_A}(-r)\right)^{2}
    = 4\delta^{2}\alpha^{2}\cos^{2}\theta
,\]
from which the first inequality follows.

If instead we take $\omega = e(-\theta)$ then we find \[
    \alpha\left(2\alpha + \omega^{2}\wh{1_A}(2r) +
        \overline{\omega}^{2}\wh{1_A}(-2r) \right)
    \geq \left( \bigl|\wh{1_A}(r)\bigr| + \bigl|\wh{1_A}(r)\bigr| \right)^2
    = \left( 2\delta\alpha \right)^{2}
\] which rearranges with the triangle inequality to give \eqref{eq:weakyudin}.

The Cauchy-Schwarz inequality $\E[XY]^2 \leq \E[X^2]\E[Y^2]$ is only close to
equality when the random variables $X$ and $Y$ are close to proportional. However, $1_A(x)$ and \[     
    1_A(x) \cdot \left( \overline{\omega}e_r(x) + \omega e_{-r}(x) \right)
    = 1_A(x) \cdot 2\cos(2 \pi r x / p + \theta)
\] are not approximately proportional, since $A$ is not thin.

Concretely, set $\omega = e(-\theta)$ again. Using the fact that
$\E[X^2] = \E[(X - \E[X])^2] + \E[X]^2$ for a random variable $X$,
we can compute
\begin{align*}
    \underset{x \in \F_p}{\E}\left[
        1_A(x)(\overline{\omega} e_{r}(x) + \omega e_{-r}(x))^2
    \right] 
    &= \alpha \underset{x \in A}{\E} \left[
        (\overline{\omega} e_{r}(x) + \omega e_{-r}(x) )^2
    \right]  \\
    &= \alpha\underset{x \in A}{\E} \left[
        (\overline{\omega}e_{r}(x) + \omega e_{-r}(x) - 2\delta)^{2}
    \right] +4\delta^{2}\alpha \\
    &= \alpha\underset{x \in A}{\E} \left[
        \left( 2\cos(2 \pi r x / p + \theta) - 2\cos\varphi \right)^2 \right]
    +4\delta^{2}\alpha \\
    &= 16\alpha\underset{x \in A}{\E} \left[
        \sin^2\left(t_1(x)\right)\sin^2\left(t_2(x)\right)
    \right] +4\delta^{2}\alpha
,\end{align*}
where $\varphi \coloneqq \arccos(\delta) \in [0 ,\pi / 2]$,
$t_1(x) \coloneqq\pi r x / p + \theta/2 + \varphi/2$ and
$t_2(x) \coloneqq \pi r x / p + \theta/2 - \varphi/2$.

We should be explicit about the fact that we are dealing with lifts
$\tilde{y} \in \mathbb{Z}$ of the elements $y = rx \in \F_p$. We can make any
choice of lift we like, so let us fix the lift so that
$\abs{\pi r x / p + \theta / 2} \leq  \pi/ 2$. It follows that \[
    \left|t_i(x)\right| \leq \pi / 2 + \varphi / 2 \leq 3\pi/4
    \] for $i= 1, 2$. Writing \[ 
    m = \frac{2\sqrt{2}}{3\pi},
\] we therefore have that\footnote{This bound can be derived by
considering the concavity of $\sin t$ in the region $0\leq t\leq 3 \pi / 4$.}
\begin{equation}\label{eq:concavitybound}
    \left|\sin(t_i(x))\right| \geq  m \left|t_i(x)\right|
.\end{equation}

Now observe that, for any $\gamma$, $\abs{t_1(x)} \leq  \gamma$ for at most
$1 + \frac{2\gamma}{\pi}p$ values of $x$. Similarly for $t_2$. We therefore
have that $t_1(x)^{2}t_2(x)^{2} \leq \gamma^{4}$ for at most
$2 + \frac{4\gamma}{\pi}p$ values of $x$. Thus
\begin{align*}
    \underset{x \in A}{\E}\left[\sin^2(t_1(x))\sin^{2}(t_2(x)) \right] 
    &\geq m^4 \underset{x \in A}{\E}
        \left[t_1(x)^{2}t_2(x)^{2}\right] \\ 
    &\geq m^4 \left(1 - \frac{4\gamma}{\alpha_0\pi}
        - \frac{2}{\alpha_0 p}\right) \gamma^{4}  \\
    &= m^4 \left(1 - \frac{4\gamma}{\alpha_0\pi}\right)\gamma^4 - o(1)
.\end{align*} 
Taking $\gamma = \frac{\pi}{5} \times \alpha_0$ makes 
$\left(1 - \frac{4\gamma}{\alpha_0\pi}\right) \gamma^{4}
    = {\alpha_0}^{4} \times \frac{\pi^{4}}{5^{5}}.$ 

Starting from \eqref{eq:yudin} we can now compute
\begin{align*} 
    \omega^{2}\wh{1_A}(2r) + \overline{\omega}^{2}\wh{1_A}(-2r)
    &= \underset{x \in \F_p}{\E}\left[ 1_A(x)\left(
        \overline{\omega}e_r(x) + \omega e_{-r}(x)
    \right)^2 \right] -2\alpha \\
    &\geq 16\alpha\underset{x \in A}{\E} \left[
        \sin^2\left(t_1(x)\right)\sin^2\left(t_2(x)\right)
    \right] +4\delta^{2}\alpha - 2\alpha \\
    &\geq 2\left(2\delta^2 - 1 + 8m^4\pi^4{\alpha_0}^4/ 5^5
        - o(1)\right)\alpha,
\end{align*}
from which the triangle inequality gives the result with
\[
    \eps = \frac{8 m^{4} \pi ^{4}}{5^{5}}{\alpha_0}^{4}
    = \frac{2^{9}}{3 ^{4} \times 5^{5}}{\alpha_0}^{4}
.\] 
\end{proof}
\begin{remarks}
    If a lower bound on $\delta$ is assumed then $\eps$ can be made slightly
    larger, by strengthening the bound in \eqref{eq:concavitybound}.

    We also have as a corollary that \[
        \bigl| \wh{1_A}(r) \bigr|
        \leq \left(1 - \Omega\left({\alpha_0}^{4}\right)
            + o_{p\rightarrow\infty}(1)\right)\alpha
    \] for any $r \neq 0$. A consequence of 
    \cite[Theorem~5]{lev2001}, is the stronger result that \[
        \bigl| \wh{1_A}(r) \bigr|
        \leq \left(1 - \Omega\left({\alpha_0}^{2}\right)
            + o_{p\rightarrow\infty}(1)\right)\alpha
    \] for any $r \neq 0$. This suggests that the factor of ${\alpha_0}^{4}$
    in $\eps$ could be replaced with a factor of ${\alpha_0}^{2}$ with some
    more work.
\end{remarks}
\subsection{Proof of Theorem \ref{thm:bound}} 

The proof of Theorem \ref{thm:bound} is a case analysis on the values of
$\wh{1_A}(r_i)$.  If $\delta_1$ and $\delta_2$ are both small, then Corollary
\ref{cor:sumfreebound} is strong enough. Otherwise, we use Proposition
\ref{prop:selfinversebound}. The question then becomes: given that $\delta_1$
is large, how small can $\sum_{i=1}^{k}\delta_i^{2}$ be under the constraints,
such as Corollary \ref{cor:sumdeltacubes}, implied by the sum-free condition? 

We will make use of the following fact for $x_1,\dots, x_n \in [0,1]$, which is
an instance of nesting of $\ell_p$-norms:
\begin{equation} \label{eq:optimisation}
    \Big(\sum_{i=1}^{n}x_i^2\Big) \geq \Big(\sum_{i=1}^{n}x_i^{3}\Big)^{{2}/{3}}
.\end{equation}

\begin{proof}[Proof of Theorem \ref{thm:bound}]
    We can assume that $\alpha \geq 0.24$, since otherwise we are done.
    We shall reason based on the value of $\delta_1$. First, we make an
    observation common to several of the cases. If we can show that there is an $h > 0$ so that 
    \[
        \sum_{i=1}^{k}\delta_i^{2}
        \geq 0.75 + h - o_{k \rightarrow \infty}(1),
   \] where the error is uniform in $A$, then applying Proposition \ref{prop:selfinversebound}
    will yield
    \begin{gather}
        \alpha \leq \frac{1}{1 + 4 \times \left(0.75 + h - o_{k \rightarrow \infty}(1)\right)} + O(k / \sqrt{p}) \nonumber \\
               < 0.25 - c_h + o_{k \rightarrow \infty}(1) + O(k / \sqrt{p})
               \label{eq:finalbound} \tag{$\dagger$}
    \end{gather}
    for some $c_h > 0$ depending only on $h$. Now, begin by 
    choosing $k$ large enough that the $o_{k \rightarrow \infty}(1)$ in
    \eqref{eq:finalbound} is less than $c_h / 3$. Then, choose $p$ large enough
    that the $O(k / \sqrt{p})$ in \eqref{eq:finalbound} is also less than
    $c_h / 3$. 
    Then $\alpha < 0.25 - c_h / 3$ as required.
    \case{Case 1: $\delta_1 \leq 0.33$} Recall the first bound from Corollary
    \ref{cor:sumfreebound}: \[
         \alpha \leq \frac{\delta_1}{1 + \delta_1}
    .\] 
    Note that as long as $\delta_1 < 1 / 3$, this is enough to bound
    $\alpha < 0.25$. In particular, 
    here we have \[
        \alpha
        \leq \frac{\delta_1}{1 + \delta_1}
        \leq \frac{ 0.33}{1.33} < 0.2482
    .\] 

    \case{Case 2: $0.33 \leq \delta_1 \leq  0.45$} Now the first conclusion of
    Corollary \ref{cor:sumfreebound} is not enough, but we can argue based on
    the value of $\delta_2$. If $\delta_2$ is small, then the second conclusion
    of Corollary \ref{cor:sumfreebound} will suffice. Otherwise, we can force
    $\sum_{i=1}^{k}\delta_i^{2}$ to be large and apply \eqref{eq:finalbound}.
    So, write $\delta_2 = a\delta_1$ where $a \in (0,1]$. 

\case{Case 2.1: $a \leq 0.7$} Apply the second conclusion of
Corollary \ref{cor:sumfreebound}, noting that the hypothesis on $\delta_1$ and
$\delta_2$ is met, to get \[
    \alpha \leq \frac{a\delta_1}{1 + a\delta_1 + 2a\delta_1^{3} - 2\delta_1^{3}}
    \leq \max_{x, y}
        \frac{xy}{1 + xy + 2x^{3}y - 2x^{3}}
,\] 
where the maximum is taken over the range $0.33 \leq x\leq 0.45$, $0\leq y\leq 0.7$.

This expression is increasing in $y$ since $x^3 \leq 1 / 2$, so \[
    \alpha
    \leq \max_{x}\frac{0.7 x}{1 + 0.7x - 0.6 x^3}
    \leq \max_x \frac{0.7 x}{1 + 0.7x - 0.6 \times 0.45 ^3 }
.\] The expression on the right hand side increases with $x$, so plugging in
$x = 0.45$ gives $\alpha < 0.24994$.

\case{Case 2.2: $a \geq 0.7$} Applying Corollary \ref{cor:sumdeltacubes} gives
\[
    \sum_{i=3}^{k} \delta_i^{3}
    \geq \frac{1}{2} - \delta_1^{3} - \delta_2^{3} - o_{k \rightarrow \infty}(1)
    = \frac{1}{2} - \left(1 + a^3\right)\delta_1^{3} - o_{k \rightarrow \infty}(1)
\] whence, by \eqref{eq:optimisation},
\begin{align}
    \sum_{i=1}^{k}\delta_i^{2}
    &\geq \left(1+a^{2}\right)\delta_1^{2} 
        + \Big(\frac{1}{2} - \left(1+a^{3}\right)\delta_1^{3}\Big)^{{2} / {3}}
        \nonumber
            - o_{k \rightarrow \infty}(1) \\
    &\geq \min_{x,y} \left( \left(1 + y^2 \right)x^2 +
    \left(\frac{1}{2} - (1+y^3)x^3 \right)^{2/3}\right) - o_{k \rightarrow \infty}(1)
    \label{eq:case2.2xy}
,\end{align} where the minimum is over the range $0.33 \leq x\leq 0.45$, $0.7 \leq y\leq 1$. One can check that the expression being minimised in \eqref{eq:case2.2xy} is increasing with $y$.
Hence
\begin{equation} \label{eq:case2.2x}
    \sum_{i=1}^{k}\delta_i^{2}
    \geq \min_x\left( 1.49x^{2} + 
        \Big(0.5 - 1.343x^{3}\Big)^{2 / 3} \right)
        - o_{k \rightarrow \infty}(1)
.\end{equation}
This new expression increases with $x$ (see Figure \ref{fig}).
So, we can compute
\[
    \sum_{i=1}^{k}\delta_i^{2} \geq 1.49 \times 0.33^{2} + \left(\frac{1}{2} - 1.343 \times 0.33^{3}\right)^{2 / 3} > 0.7510 - o_{k \rightarrow \infty}(1)
.\] 

\case{Case 3: $0.45 \leq \delta_1 \leq 0.7455$} Here $\delta_1$ is quite large,
but $\delta_1^{3} < 1 / 2$, so $\delta_2$ will have to be quite large also. This
will allow us to use \eqref{eq:finalbound}. In detail,
Corollary \ref{cor:sumdeltacubes} gives \[
    \sum_{i=2}^{k}\delta_i^{3}
    \geq \frac{1}{2} - \delta_1^{3} - o_{k \rightarrow \infty}(1)
.\] If $k$ is large enough then the right hand side is positive.
So from \eqref{eq:optimisation} we have
\begin{align} \label{eq:case3}
    \sum_{i=1}^{k}\delta_i^{2}
    &\geq \delta_1^{2} + \left(\frac{1}{2} - \delta_1^{3}\right)^{2 / 3}
        - o_{k \rightarrow \infty}(1) \nonumber \\
    &\geq \min_x \left( x^{2} + \left(\frac{1}{2} - x^{3}\right)^{2 / 3} \right)
        - o_{k \rightarrow \infty}(1) 
,\end{align}
where the minimum is taken over the range $0.45 \leq x \leq 0.7455$.
This expression is  smallest when $x = 0.7455$ (see Figure \ref{fig}).
  So we have \[
    \sum_{i=1}^{k}\delta_i^{3}
    \geq 0.7455^{2} + \left(\frac{1}{2} - 0.7455^{3}\right)^{2 / 3}
        - o_{k \rightarrow \infty}(1)
    > 0.7501 - o_{k \rightarrow \infty}(1)
.\] 

\begin{figure}
    \begin{tikzpicture}[scale=0.5]
        \begin{axis}[samples = 12,
                     smooth,
                     xlabel={$x$},
                     thick,
                     axis lines=center,
                     xmin=0.32,
                     xmax=0.46,
                     ymin=0.742,
                     ymax=0.826,
                     ]
                     \addplot[no marks,black,domain=0.32:0.47]
                         {1.49*x^2 + ((0.5 - 1.343*x^3)^(2))^(1/3)};
                     \addplot[dashed, no marks]
                         coordinates {(0.33,0) (0.33,1)};
                     \addplot[dashed, no marks]
                         coordinates {(0.45,0) (0.45,1)};
                     \addplot[no marks, red]
                         {0.75};
        \end{axis}
    \end{tikzpicture}
    \begin{tikzpicture}[scale=0.5]
        \begin{axis}[samples = 20,
                     smooth,
                     xlabel={$x$},
                     thick,
                     axis lines=center,
                     xmin=0.43,
                     xmax=0.77,
                     ymin=0.745,
                     ymax=0.797,
                     ]
                     \addplot[no marks,black,domain=0.4:0.8]
                         {x^2 + ((0.5 - x^3)^(2))^(1/3)};
                     \addplot[dashed, no marks]
                         coordinates {(0.45,0) (0.45,1)};
                     \addplot[dashed, no marks]
                         coordinates {(0.7455,0) (0.7455,1)};
                     \addplot[no marks, red]
                         {0.75};
        \end{axis}
    \end{tikzpicture}
    \begin{tikzpicture}[scale=0.5]
        \begin{axis}[samples=10,
                     smooth,
                     xlabel={$x$},
                     thick,
                     axis lines=center,
                     xmin=0.74,
                     xmax=0.815,
                     ymin = 0.745,
                     ymax=0.79
                     ]
                     \addplot[no marks,black,domain=0.74:0.8091,
                         samples at ={0.74, 0.744, ..., 0.8, 0.8, 
                            0.8001, ..., 0.8090, 0.80901}]
                            {4*x^4 - 3*x^2 + 1 + (( 0.5 + (2 * x^2 - 1)^3 - x^3)^(2))^(0.3333)};
                     \addplot[no marks,blue,domain=0.74:0.82]
                         {x^2 + ((0.5 - sqrt(2) * x^2 / 2)^2)^(1/3) };
                     \addplot[dashed, no marks] coordinates
                         {(0.7455,0) (0.7455,0.9)};
                     \addplot[dashed, no marks] coordinates
                         {(0.80901,0) (0.80901,0.9)};
                     \addplot[no marks, red] coordinates
                         {(0.6, 0.75) (0.9, 0.75)};
        \end{axis}
    \end{tikzpicture}

    \caption{ The function of $x$ which is minimised to produce a lower bound
        on $\sum_{i=1}^{k}\delta_i^{3}$ in different cases, along with the
        region on which $x$ is minimised in each case (dashed lines)
        and the constant $0.75$ (red).
        \emph{Left:} Case 2.2 given by \eqref{eq:case2.2x}.
        \emph{Centre:} Case 3 given by \eqref{eq:case3}.
        \emph{Right:} Cases 4.1 given by \eqref{eq:case4.1} (black) and
                      4.2 given by \eqref{eq:case4.2} (blue).
    }
    \label{fig}
\end{figure}
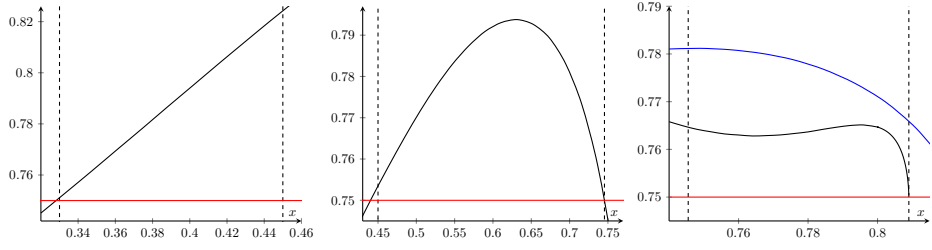

\case{Case 4: $0.7455 \leq \delta_1 \leq 0.809016$ }
If $\theta_1$ is close to $0$ or $\pi$ then Lemma \ref{lemma:positivecoeff}
will give us a large coefficient in the right half-plane. Otherwise, the
contribution of $r_1$ to Corollary \ref{cor:sumdeltacubes} is negligible.
In either case, we end up being able to use \eqref{eq:finalbound}.\footnote{
    The choice of boundary may seem odd here. The argument in this case gives
    $\alpha \leq 0.25 + o(1)$ exactly for
    $\delta_1 = \sqrt{ (3 + \sqrt{5} ) / 8 }
        \approx 0.809017$, so to get below that bound with this argument
        we consider a region slightly to the left of this critical point.}

Assume $p > 3$ and let $t$ be such that $2r_1 = \pm r_t$. Note that $t \neq 1$,
as otherwise either $2r_1 = r_1$ or $3r_1 = 0$, which both imply $r_1 = 0$ since
$p > 3$. If we write $\Delta(\delta, \theta) = 2\delta^{2}\cos^{2}\theta - 1$ for any
$\delta, \theta$, then Lemma \ref{lemma:positivecoeff} says that \[
    \Re \wh{1_A}(r_t) \geq \Delta(\delta_1, \theta_1)\alpha
.\] We also know from \eqref{eq:weakyudin} that
$\delta_t \geq 2\delta_1^{2} - 1$.

\case{Case 4.1: $\Delta(\delta_1, \theta_1) > 0$}
In this case, $\Re \wh{1_A}(r_t) > 0$. From Proposition \ref{prop:sumofcubes}
and the triangle inequality we have \[
    \delta_1^{3}\abs{\cos\theta_1} + \sum_{i \neq 1, t} \delta_i^{3}
    \geq \frac{1}{2} + \frac{\delta_t^{2}}{\alpha}\Re\wh{1_A}(r_t)
    \geq \frac{1}{2} + \left(2\delta_1^2-1\right)^{2}\Delta(\delta_1, \theta_1)
.\] 
By replacing $\theta_1$ with $\pi - \theta_1$ if necessary, we can assume
$\theta_1 \in [\pi / 2, 3 \pi / 2].$ Then
\begin{align} \label{eq:case4.1theta}
    \sum_{i \neq 1, t} \delta_i^{3}
    &\geq \frac{1}{2} + \left(2\delta_1^2-1\right)^2\Delta(\delta_1, \theta_1)
        + \delta_1^{3}\cos\theta_1 \nonumber \\
    &\geq \min_{t} \left( \frac{1}{2} +
    \left(2\delta_1^2-1\right)^2\Delta(\delta_1, t) + \delta_1^{3}\cos t \right)
,\end{align}
where the minimum is taken over the range $\pi / 2 \leq t \leq 3 \pi / 2$. 
It can be checked that this minimum is attained when $t = \pi$.
So \[
    \sum_{i \neq 1, t} \delta_i^{3}
    \geq \frac{1}{2} + \left(2\delta_1^{2} - 1\right)^{3} - \delta_1^{3}
.\] 
Then by Lemma \ref{lemma:taillemma}, since we've fixed $\alpha \geq 0.24$,
this becomes
\begin{equation}\label{eq:case4.1notheta}
    \sum_{2\leq i\leq k, i \neq t} \delta_i^{3} \geq \frac{1}{2} + \left(2\delta_1^{2} - 1\right)^3 - \delta_1^{3} - o_{k \rightarrow \infty}(1)
.\end{equation}
We can lower bound 
$\frac{1}{2} + \left( 2\delta_1^{2} - 1\right)^{3} - \delta_1^{3} > 0.000001$
here.
Therefore, by taking $k$ large enough we can
ensure that the right hand side of \eqref{eq:case4.1notheta} is positive.
It follows from \eqref{eq:optimisation} that
\begin{align} \label{eq:case4.1}
    \sum_{i=1}^{k}\delta_i^{2}
    &\geq \delta_1^2 + \left(2\delta_1^2 -1 \right)^2
        + \Big(\frac{1}{2} + \left(2\delta_1^{2}-1\right)^3
        - \delta_1^{3}\Big)^{2 / 3} - o_{k \rightarrow \infty}(1) \nonumber \\
    &\geq \min_x \left( x^2 + \left(2x^2 -1 \right)^2
        + \Big(\frac{1}{2} + \left(2x^{2}-1\right)^3
        - x^{3}\Big)^{2 / 3} \right) - o_{k \rightarrow \infty}(1) 
,\end{align} 
where the minimum is taken in the range $0.7455 \leq x \leq 0.809016$.
Now, it can be verified\footnote{
Intuitively, this sum will be smallest when all of the mass is concentrated
in $\delta_1$ and $\delta_2$, i.e\ when $\delta_1^3 - (2\delta_1^{2} - 1)^3$ is close to
$1/2$, which is when $\delta_1$ is close to $\sqrt{ ( 3 + \sqrt{5} ) / 8 } \approx 0.809017$.} that this attains its minimum when $x = 0.809016$
(see Figure \ref{fig}), so we can calculate \[
    \sum_{i=1}^{k}\delta_i^{2} > 0.75001 - o_{k \rightarrow \infty}(1)
.\] 

\case{Case 4.2: $\Delta(\delta_1, \theta_1) \leq 0$} We shall apply Corollary \ref{cor:sumdeltacubes}, which says
\[
    \sum_{i=2}^{k}\delta_i^{3} \geq \frac{1}{2} - \delta_1^{3}\abs{\cos\theta_1}
    - o_{k \rightarrow \infty}(1)
.\] 
From the assumption that $\Delta(\delta_1, \theta_1) \leq 0$ we know that $\delta_1 \left| \cos\theta_1 \right| \leq \sqrt{2} / 2$. So
\[
    \sum_{i=2}^{k}\delta_i^{3} \geq \frac{1}{2} - \frac{\sqrt{2}}{2}\delta_1^{2} - o_{k \rightarrow \infty}(1)
.\] 
Now, $1 - \delta_1^{2}\sqrt{2} \geq 1 - 0.809016 ^{2} \times \sqrt{2} > 0$ here.
So after taking $k$ large enough the right hand side
above is positive. Then applying \eqref{eq:optimisation} gives 
\begin{align}\label{eq:case4.2}
    \sum_{i=1}^{k}\delta_i^{2}
    &\geq \delta_1^{2} + \left(
        \frac{1}{2} - \frac{\sqrt{2}}{2}\delta_1^{2}
    \right)^{2 / 3} - o_{k \rightarrow \infty}(1) \nonumber \\
    &\geq \min_x \left( x^{2} + \left(
        \frac{1}{2} - \frac{\sqrt{2}}{2}x^{2}
        \right)^{2 / 3} \right)- o_{k \rightarrow \infty}(1),
\end{align}
where the minimum is taken over the range $0.7455 \leq x \leq 0.809016$.
This minimum is attained when $x = 0.809016$ (see Figure \ref{fig}).
So we can calculate
\[
    \sum_{i=1}^{k} \delta_i^{2} > 0.7659 - o_{k \rightarrow \infty}(1)
.\] 
\case{Case 5: $\delta_1 \geq 0.809016$} Here, Lemma \ref{lemma:positivecoeff}
will allow us to force $\delta_1^{2} + \delta_2^{2} > 0.750001$ and use
Proposition \ref{prop:selfinversebound}. Note that we really do
need the improvement over \eqref{eq:weakyudin}, as otherwise we get
$\delta_1^{2} + \delta_2^{2} \geq 0.75$ when
$\delta_1 = \left( \left(3 + \sqrt{5} \right) / 8 \right)^{1 / 2}$. First, take
$p$ large enough that the error in
Lemma \ref{lemma:positivecoeff} is less than $0.000001$, given
$\alpha_0 \geq 0.24$.

Then by Lemma \ref{lemma:positivecoeff} we know that
$\delta_2 \geq 2\delta_1^{2} - 1 + \eps - 0.000001$ where
\[
    \eps = \frac{2^{9}}
            {3^{4} \times 5^{5}} \times 0.24^{4} > 0.0000061
,\] 
which implies \[
    \delta_1^{2} + \delta_2^{2}
    \geq \delta_1^{2} + \left(2\delta_1^{2} -0.999994\right)^{2} 
    \geq \min_x \left( x^2 + \left( 2x^2 - 0.999994 \right)^2 \right)
,\] where the minimum is taken over the range $0.809016 \leq x \leq 1$. This
 is increasing since $x \geq 0.809016$ implies $2x^{2} > 0.999994$,
so \[
    \delta_1^{2} + \delta_2^{2} \geq 0.809016^{2} +
        \left(2 \times 0.809016^{2} - 0.999994\right)^{2} > 0.7500001
.\]

Now applying Proposition \ref{prop:selfinversebound} with $k = 2$ gives
\[
    \alpha \leq \frac{1}{1 + 4\left(\delta_1^{2} + \delta_2^{2}\right)}
        + O\left(1 / \sqrt{p}\right)
    \leq 0.249999975 + o(1)
.\] 
\end{proof}
\section{Fields of characteristic 2} \label{sec:char2}
Now suppose that $\F$ is a field of order $q = 2^{n}$, and let $A$ be a subset
of $\F^*$. Define the \emph{trace $\Tr : \F \to \F_2$} by
\[
    \Tr (x) \coloneqq \sum_{i=0}^{n-1}x^{2^{i}}
.\] 
Note that $\Tr(x) + \Tr(y) = \Tr(x+y)$.
We shall make use of the following bound on Kloosterman sums over fields of
characteristic $2$ (see \cite{conrad2002}).
\begin{lemma}\label{lemma:kloosterman2}
    If $a \in \F^{*}$ then
    \[
        \Bigl| \sum_{x \in \F^{*}}(-1)^{\Tr\left(x + ax^{-1} \right)} \Bigr|
        \leq 2\sqrt{q}
    .\] 
\end{lemma}

\begin{proof}[Proof of Proposition \ref{prop:char2}]
    Let $\gamma \from \F \to \C$ be the additive character on $\F$ given by
\[
    \gamma(x) = (-1)^{\Tr(x)}
.\] 
Define
$
X \coloneqq \F \setminus \ker \gamma
$ 
and, noting that $0 \not \in X$ since $0 \in \ker \gamma$,
$
A \coloneqq X \cap X^{-1}
.$ Then $X$ is sum-free, and $A$ is both sum-free and closed under inverses.

Note $1_X = \frac{1}{2}(1 - \gamma)$.
So, with the convention that $0^{-1} = 0$, we have
\begin{align*}
    \alpha 
    = \underset{x}{\E}\left[ 1_X(x) 1_{X^{-1}}(x) \right]  
    &= \underset{x}{\E}\left[ 1_X(x) 1_X(x^{-1}) \right] \\
    &= \frac{1}{4}\underset{x}{\E}\left[ (1 - \gamma(x))(1 - \gamma(x^{-1})) \right] \\
    &= \frac{1}{4}
       +\frac{1}{4}\underset{x}{\E}\left[\gamma(x) \gamma(x^{-1})\right] \\
.\end{align*}
Since $\Tr(x) + \Tr(x^{-1}) = \Tr(x+x^{-1})$, we have $\gamma(x)\gamma(x^{-1}) = \gamma(x + x^{-1})$. Then \[
    \bigl| \underset{x}{\E}\left[ \gamma(x) + \gamma(x^{-1})) \right] \bigr| 
    = \bigl| \underset{x}{\E}\left[ \gamma(x + x^{-1}) \right] \bigr|
    \leq \frac{2\sqrt{q}}{q}
    = o(1)
\] by Lemma \ref{lemma:kloosterman2}, which gives our result.
\end{proof}

\section{Final remarks} \label{sec:remarks}

\subsection{} Write $\sigma(\F)$ for the density $\abs{A}/\abs{\F}$ of the
largest sum-free subset $A$ of $\F$. This quantity was studied in the more
general context of finite Abelian groups by Diananda and Yap in \cite{dy1969}.
Recall from Section \ref{sec:introduction} that we define $\mu(\F)$ to be the
density of the largest subset of $\F$ which is both sum-free and closed under
inverses.

When $\F$ has characteristic $2$ it can be seen that $\sigma(\F) = 1/2$, as the
set $X$ in the proof of Proposition \ref{prop:char2} demonstrates. Moreover,
Proposition \ref{prop:char2} itself shows $\mu(\F) \geq 1/4 - o(1)$.

When $\F$ has prime order $p>2$, the interval
$I = \{ x \in \F : p/3 < x < 2 p/3 \}$ has density $1/3 + o(1)$, and this is
the best possible by the Cauchy-Davenport inequality. As described
in \cite[p.~8]{bhs2019}, the set $I \cap I^{-1}$ is then sum-free and closed
under inverses, and has density $1/9 - o(1)$. So $\mu(\F) \geq 1/9 - o(1)$.

It is reasonable to suspect that the events `$A$ is sum-free' and
`$A^{-1}$ is sum-free' are independent. So, we conjecture that the lower bounds
above are in fact tight:
\begin{conjecture}
    Let $\F$ be a finite field. Then $\mu(\F) = \sigma(\F)^{2} + o(1)$
    as $\abs{\F} \rightarrow \infty$.
\end{conjecture}
\subsection{} For a set $A \subseteq \F^{*}$ we can use the quantity
\[
    I(A) \coloneqq \frac{\abs{A \cap A^{-1}}}{\abs{A}}
\] 
to measure `how much' $A$ is closed under inverses.
So we have studied sum-free sets $A$ with $I(A) = 1$.
When $\F$ has prime order $p$ and $A$ is sum-free with $I(A)$ large, we might
still expect to do better than the bound of $\abs{A} < (p+1) / 3$ given by the
Cauchy-Davenport inequality. Indeed, since $A \cap A^{-1}$ is itself sum-free
and closed under inverses we have \[
    \alpha = \abs{A}/p= \frac{\abs{A \cap A^{-1}}}{I(A) \times p} 
    \leq \frac{\mu(\F)}{I(A)}
.\] 
So when $I(A) \geq  0.75$ we can use Theorem \ref{thm:bound} to deduce
\[
    \alpha \leq \frac{\mu(\F)}{0.75}
    \leq \frac{\left(0.25 - c\right) + o(1)}{0.75}
    \leq  \left(1 - 4c\right) / 3 + o(1)
.\] 

\section*{Acknowledgements}

This work was funded by a London Mathematical Society Undergraduate Research
Bursary and the Mathematical Institute at the University of Oxford. I am
immensely grateful to Tom Sanders for suggesting this topic of research and for
his unwaveringly enthusiastic mentorship throughout the summer. His mathematical
advice was invaluable, as were his many helpful comments on the drafts of this
paper.

\end{document}